\documentclass[review]{elsarticle}

\usepackage{lineno,hyperref}
\usepackage[fleqn]{amsmath}
\usepackage{mathptmx}
\usepackage{amsfonts}
\DeclareMathAlphabet{\mathcal}{OMS}{cmsy}{m}{n}
\usepackage{enumerate}
\usepackage{titletoc}
\usepackage{amssymb,mathrsfs, amsthm,enumerate}
\theoremstyle{plain}
\newtheorem{thm}{Theorem}[section]

\newtheorem{prop}[thm]{Proposition}

\theoremstyle{definition}
\newtheorem{defn}{Definition}[section]

\modulolinenumbers[5]

\journal{Indagationes Mathematicae}

\begin{document}

\begin{frontmatter}

\title{Implication-Based Intuitionistic Fuzzy Finite State Machine over a Finite Group}


\author[mymainaddress]{M. Selvarathi \corref{mycorrespondingauthor}}
\cortext[mycorrespondingauthor]{Corresponding author}
\ead{selvarathi.maths@gmail.com}


\address[mymainaddress]{Department of Mathematics, Karunya Institute of Technology and Sciences, Coimbatore, India}

\begin{abstract}
In this paper \emph{implication-based intuitionistic fuzzy finite state machine} otherwise called as \emph{Implication-based intuitionistic fuzzy semiautomaton} (IB-IFSA) over a finite group is defined and investigated intensively. The abstraction of \emph{implication-based intuitionistic fuzzy kernel} and \emph{implication-based intuitionistic fuzzy subsemiautomaton} of an IB-IFSA over a finite group are brought about using the notion of \emph{implication-based intuitionistic fuzzy subgroup} and \emph{implication-based intuitionistic fuzzy normal subgroup}. The necessary and sufficient condition for an \emph{implication-based intuitionistic fuzzy normal subgroup}  to be an \emph{implication-based intuitionistic fuzzy kernel}  as well as the necessary and sufficient condition for an \emph{implication-based intuitionistic fuzzy subgroup} to be an \emph{implication-based intuitionistic fuzzy subsemiautomaton} of an IB-IFSA  are studied. 
\end{abstract}


\end{frontmatter}

\section{Introduction}
\par In 1965, Zadeh \cite{Zad65} first launched the concept of \emph{fuzzy} set. Later it was Wee \cite{Wee69} who perceived the notion of \emph{fuzzy automata} in 1969. In 1971, Rosenfeld  \cite{Ros71} elaborated the study with the application of the concept of fuzzy sets initiated by Zadeh to groups. It led to the induction of many research works that were launched on their algebraic structures. Further seeral studies had also been done on the \emph{fuzzy normal subgroups} by \cite{Dib98}, \cite{Mal92} and \cite{Muk84}. Asok Kumar \cite{Aso99} has made a study about the products of fuzzy subgroups. In 1986, a generalization on the concept of fuzzy set proposed by Zadeh was presented by Atanasssov \cite{Ant86} along with the introduction on the notion of \emph{intuitionistic fuzzy set}. The concpet of \emph{intuitionisitc fuzzy subgroup} and \emph{intuitionistic fuzzy normal subgroup} was asserted by Li Xiaoping \cite{LiX00}. A study on the algebraic techniques of the fuzzy finite state machine in fuzzy automata theory was prompted by Malik  \cite{Mal97} in 1997. Hofer \cite{Hof87} and Fong \cite{Fon88} inducted a study on group semiautomaton. P. Das \cite{Das97} fuzzified the notion of group semiautomaton he insinuated the concept of fuzzy semiautomaton over a finite group. An assertion on the \emph{implication-based} \emph{fuzzy subgroup} was proposed by Yuan \cite{Yua03}. A proposition on the \emph{implication-based fuzzy normal subgroup}, \emph{implication-based fuzzy semiautomaton} (IBFSA) and \emph{implication-based intuitionistic fuzzy subgroup} over a finite group have been postulated by M. Selva Rathi \cite{Rat15} \cite{Rat16}  and  \cite{Rat17} in 2015. The research work in this paper involves further study in the above field and the introuduction of the concept of \emph{implication-based intuitionistic fuzzy semiautomaton} of a finite group with the use of \emph{implication-based intuitionistic fuzzy subgroup}. The paper also exemplifies the \emph{implication-based intuitionistic fuzzy kernel} and \emph{implication-based intuitionistic fuzzy subsemiautomaton} of an \emph{implication-based intuitionistic fuzzy semiautomaton} and attempts to prove few properties.

\section{Preliminaries}
\begin{defn} \cite{Ros71}
Let $(\Omega,\cdot)$ be a group. Let a fuzzy set in $\Omega$ be a function $A$ from $\Omega$ to $[0,1]$. $A$ will be called a \textit{fuzzy subgroup} of $\Omega$, if for all ${\xi}_1, {\xi}_2$ in $\Omega$, $A({\xi}_1{\xi}_2) \geq min\big( A({\xi}_1), A({\xi}_2)\big) $ and $A({{\xi}_1}^{-1}) \geq A({\xi}_1)$
\end{defn}

\begin{defn}\cite{Ant86} 
Let $X$ be a nonempty classical set. The traid formed as \\ $A = \{ \langle \xi, \mu_A(\xi),$ $ \nu_A(\xi)\rangle \vert \xi \in X \} $ on $X$ is called an \emph{intuitionistic fuzzy set} on $X$, where the functions $\mu_A: A \rightarrow (0,1)$ and $\nu_A : A \rightarrow (0,1) $ denotes the degree of membership and the degree of nonmembership of each element $\xi \in X$ to the set $A$ respectively and $0 \leq \mu_A(\xi) + \nu_A(\xi) \leq 1$
\end{defn}

\begin{defn} \cite{LiXWG00}
Let $\Omega$ be a classical group, \emph{the intuitionistic fuzzy subset} \\ $A = \{ \langle \xi, \mu_A(\xi),$ $ \nu_A(\xi)\rangle \vert \xi \in \Omega \} $ is called an \emph{intuitionistic fuzzy group} on $\Omega$, if the following conditions are satisfied.
\begin{enumerate}[(i)]
	\item $ \mu_A(\xi \psi) \geq min\{ \mu_A(\xi), \mu_A(\psi)\}$, $\nu_A(\xi \psi) \leq max\{ \nu_A(\xi), \nu_A(\psi)\}$ , for all $\xi, \psi  \in \Omega$
	\item $\mu_A({\xi}^{-1}) \geq \mu_A(\xi)$,  $\nu_A({\xi}^{-1}) \leq \nu_A(\xi)$, for all $\xi \in \Omega$
\end{enumerate}
\end{defn}

\begin{defn}\cite{LiX00}
Let $\Omega$ be a classical group,  $A = \{ \langle \xi, \mu_A(\xi), \nu_A(\xi)\rangle \vert \xi \in \Omega \} $ be an \emph{intuitionistic fuzzy set} on $\Omega$, then $A$ is called \emph{intuitionistic fuzzy normal subgroup} on $\Omega$ if $\mu_A( \xi \psi {\xi}^{-1}) \geq \mu_A(\psi)$, $\nu_A(\xi \psi {\xi}^{-1}) \leq \nu_A(\psi)$ for all $\xi, \psi \in \Omega$
\end{defn}

\par Let $X$ be an universe of discourse and $(\Omega, \cdot)$ be a group. In \emph{fuzzy} logic, $[\alpha]$ is used to denote the truth value of \emph{fuzzy} proposition $\alpha$. The \emph{fuzzy} logical and the corresponding set theoretical notations used in this paper are \\$(\xi \in A) = A(\xi);$\\ $(\alpha \wedge \beta) = min\{ [\alpha], [\beta]\};$\\  $(\alpha \rightarrow \beta) = min\{1, 1 - [\alpha]+ [\beta] \} ;$\\ $(\forall  \xi \mbox{  } \alpha(\xi)) = {inf}_{\xi \in X}[\alpha(\xi)];$\\ $(\exists  \xi \mbox{  } \alpha(\xi)) = {sup}_{\xi \in X}[\alpha(\xi)];$ and  \\$\vDash \alpha \mbox{ if and only if } [\alpha] = 1 \mbox{ for all valuations.}$ \\

The truth valuation rules used here are that of the \L ukasiewicz system of continuous-valued logic.\\

The concept of $\lambda $ - tautology is $\vDash_{\lambda}\alpha$ if and only if $ [\alpha] \geq \lambda$ for all valuation by Ying \cite{Yin91}.

\begin{defn} \cite{Yua03}
Let $A$ be a \textit{fuzzy} subset of a finite group $\Omega$ and $\lambda \in \left. \left(0,1\right.\right]$ is a fixed number. If for any ${\xi}_1, {\xi}_2 \in \Omega$, $\vDash_{\lambda}({\xi}_1 \in A)\wedge ({\xi}_2 \in A) \rightarrow ({\xi}_1 {\xi}_2 \in A)$ and $\vDash_{\lambda} ({\xi}_1 \in A) \rightarrow ({{\xi}_1}^{-1} \in A)$.
Then $A$ is called an \textit{implication-based fuzzy subgroup} of $\Omega$.
\end{defn}

\begin{defn} \cite{Rat15}
Let $A$ be an \emph{implication-based fuzzy subgroup} of $\Omega$, $\lambda \in \left. \left(0,1\right.\right]$ is a fixed number and $f: \Omega \rightarrow \Omega$ be a function defined on $\Omega$. Then the \emph{implication-based fuzzy subgroup} $B$ of $f(\Omega)$ is defined by $\vDash_{\lambda}\left( \exists \xi \{ (\xi \in A)\}; \xi \in f^{-1}(\psi) \right) \rightarrow (\psi \in B)$, for all $\psi \in f(\Omega)$ \\

Similarly if $B$ is an \emph{implication-based fuzzy subgroup} of $f(\Omega)$ then the \emph{implication-based fuzzy subgroup} $A = f \circ B$ in $\Omega$ is defined as $\vDash_{\lambda}(f(\xi) \in B) \rightarrow (\xi \in A)$ for all $\xi \in \Omega$ and is called the pre-image of $B$ under $f$.
\end{defn}

\begin{defn} \cite{Rat15}
An \emph{implication-based fuzzy subgroup} $A$ of $\Omega$ is called an \emph{implication-based fuzzy normal subgroup} if \,
$\vDash_{\lambda} (\xi \psi \in A) \rightarrow (\psi \xi \in A) \quad \forall \xi, \psi \in \Omega$ where $\lambda \in \left. \left(0,1\right.\right]$ is a fixed number.
\end{defn}

\begin{prop} \cite{Rat15}
\label{Proposition 2.1}
Let A be an implication-based fuzzy subgroup of a finite group $\Omega$ then for any $\xi \in \Omega$, $\vDash_{\lambda}(\xi \in A) \rightarrow (\epsilon \in A)$ where $\epsilon$ is the identity element of the group $\Omega$.
\end{prop}
Hereafter let $\Omega$ be a finite group with the identity element $'\epsilon'$ and $\lambda \in \left. \left(0,1\right.\right]$ is a fixed number.

\section{Implication-Based Intuitionistic Fuzzy Semiautomaton over a Finite Group}
\begin{defn}
Let ${\mathbb{IF}}_{\mathscr{A}} = \langle A,B \rangle$ be an \emph{implication-based intuitionistic fuzzy subgroup} over a finite group $\Omega$. An \emph{implication-based intuitionistic fuzzy semiautomaton} over the finite group $(\Omega, \cdot)$ is a triple $\mathbb{IS} = \left(\Omega, \Delta, {\mathbb{IF}}_{\mathscr{A}} \right)$ where $\Delta$ denotes the set of all logic variables. 
\end{defn}
Let $\Delta^{\ast}$ denote the set of all combination of these logic variables along with the \textbf{0} function.
\begin{defn}
Let $\mathbb{IS} = \left(\Omega, \Delta, {\mathbb{IF}}_{\mathscr{A}} \right)$ be an \emph{implication-based intuitionistic fuzzy semiautomaton} over the finite group $\Omega$. Define ${\mathbb{IF}}_{\mathscr{A}^{\ast}} = \langle A^{\ast},B^{\ast} \rangle$ in $\Omega \times \Delta^{\ast} \times \Omega $ by
\begin{equation}
\begin{aligned}
&\vDash_{\lambda}((\alpha,\textbf{0},\beta) \in A^{\ast}) \rightarrow 0 \quad \quad \left(Here \mbox{  } \lambda = 0 \right) \\ \nonumber
&\vDash_{\lambda}((\alpha,\textbf{0},\beta) \in B^{\ast}) \rightarrow 1 \\ \nonumber
&\vDash_{\lambda}\left(\exists \gamma \{((\beta,\xi,\gamma)\in A^{\ast}) \wedge ((\gamma,\omega,\alpha) \in A^{\ast}) \}; \gamma \in \Omega \right) \rightarrow ((\beta, \xi \odot \omega, \alpha) \in A^{\ast})  \\
&\vDash_{\lambda} \left( \forall \gamma \{ ((\beta, \xi, \gamma)\in B^{\ast}) \vee ((\gamma, \omega, \alpha) \in B^{\ast}) \};  \gamma \in \Omega \right) \rightarrow ((\beta, \xi \odot \omega, \alpha) \in B^{\ast}) \\
&\mbox{ } \hspace{6.5 cm} \mbox{for all} \quad \alpha, \beta \in \Omega; \xi \in \Delta^{\ast}; \omega \in \Delta \nonumber
\end{aligned}
\end{equation}
\end{defn}

\begin{thm}
Let $\mathbb{IS} = \left(\Omega, \Delta, {\mathbb{IF}}_{\mathscr{A}} \right)$ be an implication-based intuitionistic fuzzy semiautomaton over the finite group $\Omega$. Then 
\begin{equation}
\begin{aligned}
&\vDash_{\lambda}\left(\exists \gamma \{((\beta, \xi,\gamma)\in A^{\ast}) \wedge ((\gamma,\psi,\alpha) \in A^{\ast}) \}; \gamma \in \Omega \right) \rightarrow ((\beta, \xi \odot \psi, \alpha) \in A^{\ast}) \\
&\vDash_{\lambda} \left( \forall \gamma \{ ((\beta,\xi,\gamma)\in B^{\ast}) \vee ((\gamma,\psi,\alpha) \in B^{\ast}) \};  \gamma \in \Omega \right) \rightarrow ((\beta, \xi \odot \psi, \alpha) \in B^{\ast}) \\
&\mbox{ } \hspace{6.5 cm} \mbox{for all} \quad \alpha, \beta \in \Omega; \xi, \psi \in \Delta^{\ast} \nonumber
\end{aligned}
\end{equation}
\end{thm}
\begin{proof}
Let $\alpha, \beta \in \Omega$ and $ \xi, \psi \in \Delta^{\ast} $. \\
The theorem is proved by the method of induction on the $ord(\psi) = n$.
If $n = 0$ then $\psi = \textbf{0}$. 
$ \xi \odot \psi = \xi \odot \textbf{0} = \xi$.
\begin{equation}
\begin{aligned}
&\vDash_{\lambda}\left(\exists \gamma \{((\beta,\xi, \gamma)\in A^{\ast}) \wedge ((\gamma, \psi, \alpha) \in A^{\ast}) \}; \gamma \in \Omega \right) \\ \nonumber
&\rightarrow \left(\exists \gamma \{((\beta, \xi, \gamma)\in A^{\ast}) \wedge ((\gamma,\textbf{0},\alpha) \in A^{\ast}) \}; \gamma \in \Omega \right) \\ \nonumber
&\rightarrow ((\beta,\xi,\alpha) \in A^{\ast}) \\ \nonumber
&\rightarrow ((\beta, \xi \odot \psi, \alpha) \in A^{\ast})  \nonumber
\end{aligned}
\end{equation}
\begin{equation}
\begin{aligned}
&\vDash_{\lambda} \left( \forall \gamma \{ ((\beta,\xi,\gamma)\in B^{\ast}) \vee ((\gamma,\psi,\alpha) \in B^{\ast}) \};  \gamma \in \Omega \right)  \\ \nonumber
&\rightarrow \left( \forall \gamma \{ ((\beta, \xi, \gamma)\in B^{\ast}) \vee ((\gamma,\textbf{0},\alpha) \in B^{\ast}) \};  \gamma \in \Omega \right)  \\ \nonumber
&\rightarrow ((\beta, \xi, \alpha) \in B^{\ast}) \\ \nonumber
&\rightarrow ((\beta, \xi \odot \psi, \alpha) \in B^{\ast})  \nonumber
\end{aligned}
\end{equation}
Therefore the result is true for $n = 0$.\\
Assume that the result hold for $\zeta \in \Delta^{\ast} $ such that $ord(\zeta) = n - 1$ and $n > 0$.
Let  $\psi \in \Delta^{\ast} $ such that $\psi = \zeta \odot \omega$ with $\zeta \in \Delta^{\ast} $, $\omega \in \Delta $,  $ord(\zeta) = n - 1$ and $n > 0$.
\begin{equation}
\begin{aligned}
&\vDash_{\lambda}\left(\exists \gamma \{((\beta,\xi,\gamma)\in A^{\ast}) \wedge ((\gamma,\psi,\alpha) \in A^{\ast}) \}; \gamma \in \Omega \right) \\ \nonumber
&\rightarrow \left(\exists \gamma \{((\beta,\xi,\gamma)\in A^{\ast}) \wedge ((\gamma,\zeta \odot \omega,\alpha) \in A^{\ast}) \}; \gamma \in \Omega \right)  \\ \nonumber
&\rightarrow \left(\exists \gamma \{((\beta,\xi,\gamma)\in A^{\ast}) \wedge \left(\exists \delta \{ ((\gamma,\zeta,\delta) \in A^{\ast}) \wedge ((\delta,\omega,\alpha) \in A^{\ast})  \}; \delta \in \Omega \right) \}; \gamma \in \Omega \right)\\
&\rightarrow \left(\exists \gamma, \delta \{ ((\beta,\xi,\gamma)\in A^{\ast}) \wedge ((\gamma,\zeta,\delta) \in A^{\ast}) \wedge  ((\delta,\omega,\alpha) \in A^{\ast})  \}; \gamma, \delta \in \Omega \right)  \\
&\rightarrow \left(\exists \delta \{ \left( \exists \gamma\{ ((\beta,\xi,\gamma)\in A^{\ast}) \wedge ((\gamma,\zeta,\delta) \in A^{\ast}) \};  \gamma \in \Omega \right) \wedge ((\delta,\omega,\alpha) \in A^{\ast}) \}; \delta \in \Omega \right) \\
&\rightarrow \left( \exists \delta \{ ((\beta, \xi \odot \zeta, \delta) \in A^{\ast})  \wedge ((\delta,\omega,\alpha) \in A^{\ast}) \}; \delta \in \Omega \right) \\ \nonumber
&\rightarrow ((\beta ,\xi \odot \zeta \odot \omega, \alpha) \in A^{\ast}) \\ \nonumber
&\rightarrow ((\beta ,\xi \odot \psi, \alpha) \in A^{\ast})  \nonumber
\end{aligned}
\end{equation}
\begin{equation}
\begin{aligned}
&\vDash_{\lambda}\left(\forall \gamma \{((\beta,\xi,\gamma)\in B^{\ast}) \vee ((\gamma,\psi,\alpha) \in B^{\ast}) \}; \gamma \in \Omega \right) \\ \nonumber
&\rightarrow \left(\forall \gamma \{((\beta,\xi,\gamma)\in B^{\ast}) \vee ((\gamma,\zeta \odot \omega,\alpha) \in B^{\ast}) \}; \gamma \in \Omega \right)  \\ \nonumber
&\rightarrow \left(\forall \gamma \{((\beta,\xi,\gamma)\in B^{\ast}) \vee \left(\forall \delta \{ ((\gamma,\zeta,\delta) \in B^{\ast}) \vee ((\delta,\omega,\alpha) \in B^{\ast})  \}; \delta \in \Omega \right) \}; \gamma \in \Omega \right) \\
&\rightarrow \left(\forall \gamma, \delta \{ ((\beta,\xi,\gamma)\in B^{\ast}) \vee ((\gamma,\zeta,\delta) \in A^{\ast}) \vee ((\delta,\omega,\alpha) \in B^{\ast})  \}; \gamma, \delta \in \Omega \right) \\
&\rightarrow \left(\forall \delta \{ \left( \forall \gamma\{ ((\beta,\xi,\gamma)\in B^{\ast}) \vee ((\gamma,\zeta,\delta) \in B^{\ast}) \}; \gamma \in \Omega \right) \vee ((\delta,\omega,\alpha) \in B^{\ast}) \}; \delta \in \Omega \right) \\
&\rightarrow \left( \forall \delta \{ ((\beta, \xi \odot \zeta, \delta) \in B^{\ast})  \vee ((\delta,\omega,\alpha) \in B^{\ast}) \}; \delta \in \Omega \right) \\ \nonumber
&\rightarrow ((\beta ,\xi \odot \zeta \odot \omega, \alpha) \in B^{\ast}) \\ \nonumber
&\rightarrow ((\beta ,\xi \odot \psi, \alpha) \in B^{\ast})  \nonumber
\end{aligned}
\end{equation}
Hence the result is valid for $ord(\psi) = n$. This completes the proof. 
\end{proof}
\begin{defn}
Let ${\mathbb{IF}}_{\mathscr{A}} = \langle A,B \rangle$ be an \emph{implication-based intuitionistic fuzzy subgroup} over a finite group $(\Omega,.)$ and let $\mathbb{IS} = \left(\Omega, \Delta, {\mathbb{IF}}_{\mathscr{A}} \right)$ be an \emph{implication-based intuitionistic fuzzy semiautomaton} over the finite group $(\Omega,\cdot)$. Then an \emph{intuitionistic fuzzy subset} $\langle \iota, \overline{A}, \overline{B} \rangle$ of the group $\Omega$ is called an \emph{implication-based} \emph{intuitionistic fuzzy} \emph{subsemiautomaton} of $\mathbb{IS}$ if 
\begin{enumerate}[(i)]
\item ${\mathbb{IF}}_{\tilde{\mathscr{A}}} = \langle \overline{A}, \overline{B} \rangle$  is an \emph{implication-based intuitionistic fuzzy subgroup} of $\Omega$. 
\item $\vDash_{\lambda} ((\alpha,\xi,\beta) \in A) \wedge (\alpha \in \overline{A}) \rightarrow (\beta \in \overline{A}) $ 
\item $\vDash_{\lambda} (\beta \in \overline{B}) \rightarrow  ((\alpha,\xi,\beta) \in B) \vee (\alpha \in \overline{B})$ 
\item [ ] \mbox{ } \hspace{ 2 cm} \mbox{ for all } $ \alpha, \beta \in \Omega $ , $ \xi \in \Delta $
\end{enumerate}
\end{defn}
\begin{defn}
Let ${\mathbb{IF}}_{\mathscr{A}} = \langle A,B \rangle$ be an \emph{implication-based intuitionistic fuzzy subgroup} over a finite group $(\Omega,\cdot)$ and let $\mathbb{IS} = \left(\Omega, \Delta, {\mathbb{IF}}_{\mathscr{A}} \right)$ be an \emph{implication-based intuitionistic fuzzy semiautomaton} over the finite group $(\Omega,\cdot)$. Then an \emph{intuitionistic fuzzy subset} $\langle \iota, \overline{A}, \overline{B} \rangle$ of the group $\Omega$ is called an \emph{implication-based} \emph{intuitionistic fuzzy} \emph{kernel} of $\mathbb{IS}$ if 
\begin{enumerate}[(i)]
\item ${\mathbb{IF}}_{\tilde{\mathscr{A}}} = \langle \overline{A}, \overline{B} \rangle$  is an \emph{implication-based intuitionistic fuzzy normal subgroup} of $\Omega$. 
\item $\vDash_{\lambda} ((\beta \kappa, \xi,\alpha) \in A) \wedge ((\beta, \xi, \gamma) \in A) \wedge (\kappa \in \overline{A}) \rightarrow (\alpha \gamma^{-1} \in \overline{A}) $ 
\item $\vDash_{\lambda} (\alpha \gamma^{-1} \in \overline{B}) \rightarrow  ((\beta \kappa,\xi, \alpha) \in B) \vee ((\beta, \xi, \gamma) \in B) \vee (\kappa \in \overline{B})$ 
\item[ ] $ \mbox{ } \hspace{3 cm} \mbox{ for all }  \alpha, \beta, \gamma, \kappa \in \Omega $ , $ \xi \in \Delta $
\end{enumerate}
\end{defn}
\begin{thm}
Let $\mathbb{IS} = \left(\Omega, \Delta, {\mathbb{IF}}_{\mathscr{A}} \right)$ be an implication-based intuitionistic fuzzy semiautomaton over the finite group $(\Omega,\cdot)$. Let ${\mathbb{IF}}_{\tilde{\mathscr{A}}} = \langle \overline{A}, \overline{B} \rangle$ be an implication-based intuitionistic fuzzy subgroup of $\Omega$. Then ${\mathbb{IF}}_{\tilde{\mathscr{A}}}$ is an implication-based intuitionistic fuzzy subsemiautomaton of $\mathbb{IS}$ if and only if 
\begin{equation}
\begin{aligned}
&\vDash_{\lambda} ((\alpha,\xi,\beta) \in A^{\ast}) \wedge (\alpha \in \overline{A}) \rightarrow (\beta \in \overline{A})  \\ \nonumber
&\vDash_{\lambda} (\beta \in \overline{B})  \rightarrow  ((\alpha, \xi, \beta) \in B^{\ast}) \vee (\alpha \in \overline{B}) 
&\mbox{ } \hspace{ 2cm} \mbox{ for all } \alpha, \beta \in \Omega  ,  \xi \in {\Delta}^{\ast} 
\end{aligned}
\end{equation}
\end{thm}
\begin{proof}
Let ${\mathbb{IF}}_{\tilde{\mathscr{A}}} = \langle \overline{A}, \overline{B} \rangle$ be an \emph{implication-based} \emph{intuitionistic fuzzy} \emph{subsemiautomaton} of $\mathbb{IS}$.\\
Let $\alpha, \beta \in \Omega $ , $ \xi \in {\Delta}^{\ast} $.
The proof is by the method of induction on $ord(\xi) = n$.
If $n = 0$ then $\xi = \textbf{0}$. 
\begin{equation}
\begin{aligned}
\vDash_{\lambda} ((\alpha,\xi,\beta) \in A^{\ast}) \wedge (\alpha \in \overline{A}) 
& \rightarrow ((\alpha,\textbf{0},\beta) \in A^{\ast}) \wedge (\alpha \in \overline{A}) \\ \nonumber
&\rightarrow (0 \in A^{\ast}) \wedge (\alpha \in \overline{A}) \\ \nonumber
&\rightarrow (0 \in A^{\ast}) \\ \nonumber
&\rightarrow (\beta \in \overline{A})  \nonumber
\end{aligned}
\end{equation}
\begin{equation}
\begin{aligned}
\vDash_{\lambda} (\beta \in \overline{B}) 
&\rightarrow 1 \\ \nonumber
&\rightarrow 1 \vee  (\alpha \in \overline{B}) \\ \nonumber
&\rightarrow ((\alpha,\textbf{0},\beta) \in B^{\ast}) \vee  (\alpha \in \overline{B}) \\ \nonumber
&\rightarrow ((\alpha, \xi, \beta) \in B^{\ast}) \vee  (\alpha \in \overline{B}) \nonumber
\end{aligned}
\end{equation}
Therefore the result is true for $n = 0$.
Assume that the result is true for all $\psi \in {\Delta}^{\ast} $ such that $ord(\psi) = n - 1$ and $n > 0$. 
Let $\xi = \psi \odot \omega$ where $\omega \in \Delta $. Then
\begin{equation}
\begin{aligned}
\vDash_{\lambda} ((\alpha,\xi,\beta) \in A^{\ast}) \wedge (\alpha \in \overline{A}) 
& \rightarrow ((\alpha, \psi \odot \omega ,\beta) \in A^{\ast}) \wedge (\alpha \in \overline{A}) \\ \nonumber
& \rightarrow \left( \exists \gamma \{((\alpha,\psi,\gamma)\in A^{\ast}) \wedge ((\gamma,\omega,\beta) \in A) \}; \gamma \in \Omega \right) \wedge (\alpha \in \overline{A}) \\ \nonumber
&\rightarrow \left( \exists \gamma \{((\alpha, \psi, \gamma)\in A^{\ast}) \wedge ((\gamma,\omega,\beta) \in A) \wedge  (\alpha \in \overline{A}) \}; \gamma \in \Omega \right) \\
& \rightarrow \left( \exists \gamma \{((\alpha,\psi,\gamma) \in A^{\ast}) \wedge (\alpha \in \overline{A}) \wedge ((\gamma,\omega,\beta) \in A)  \}; \gamma \in \Omega \right) \\ \nonumber
& \rightarrow \left( \exists \gamma \{ (\gamma \in \overline{A}) \wedge ((\gamma,\omega,\beta) \in A)  \}; \gamma \in \Omega \right)  \\ \nonumber
&\rightarrow \left( \exists \gamma \{((\gamma,\omega,\beta) \in A) \wedge (\gamma \in \overline{A}) \}; \gamma \in \Omega\right)  \\ \nonumber
& \rightarrow (\beta \in \overline{A}) \nonumber
\end{aligned}
\end{equation}
\begin{equation}
\begin{aligned}
\vDash_{\lambda} (\beta \in \overline{B})
&\rightarrow \left( \forall \gamma \{((\gamma,\omega,\beta) \in B) \vee (\gamma \in \overline{B}) \}; \gamma \in \Omega \right)  \\ \nonumber
&\rightarrow \left( \forall \gamma \{ ((\alpha,\psi,\gamma)\in B^{\ast}) \vee (\alpha \in \overline{B}) \vee ((\gamma,\omega,\beta) \in B) \}; \gamma \in \Omega \right) \\ \nonumber
&\rightarrow \left( \forall \gamma \{ ((\alpha,\psi,\gamma)\in B^{\ast})  \vee ((\gamma,\omega,\beta) \in B) \};  \gamma \in \Omega \right) \vee (\alpha \in \overline{B})\\
&\rightarrow ((\alpha,\psi \odot \omega ,\beta) \in B^{\ast}) \vee (\alpha \in \overline{B}) \\ \nonumber
&\rightarrow ((\alpha,\xi,\beta) \in B^{\ast}) \vee (\alpha \in \overline{B}) \nonumber 
\end{aligned}
\end{equation}
The converse is trivial.
\end{proof}
\begin{thm}
Let $\mathbb{IS} = \left(\Omega, \Delta, {\mathbb{IF}}_{\mathscr{A}} \right)$ be an implication-based intuitionistic fuzzy semiautomaton over the finite group $(\Omega,\cdot)$. Let ${\mathbb{IF}}_{\tilde{\mathscr{A}}} = \langle \overline{A}, \overline{B} \rangle$ be an implication-based intuitionistic fuzzy subgroup of $\Omega$. Then ${\mathbb{IF}}_{\tilde{\mathscr{A}}}$ is an implication-based intuitionistic fuzzy kernel of $\mathbb{IS}$ if and only if  
\begin{equation}
\begin{aligned}
&\mbox{(i)} \vDash_{\lambda} ((\beta \kappa,\xi,\alpha) \in A^{\ast}) \wedge ((\beta,\xi,\gamma) \in A^{\ast}) \wedge (\kappa \in \overline{A}) \rightarrow (\alpha \gamma^{-1} \in \overline{A}) \\ \nonumber
&\mbox{(ii)}\vDash_{\lambda} (\alpha \gamma^{-1} \in \overline{B}) \rightarrow  ((\beta \kappa, \xi,\alpha) \in B^{\ast}) \vee ((\beta,\xi,\gamma) \in B^{\ast})  \vee (\kappa \in \overline{B}) \\ \nonumber
&\mbox{ } \hspace{ 1.5 cm}   \mbox{ for all }  \alpha, \beta, \gamma, \kappa \in \Omega  ,  \xi \in {\Delta}^{\ast} \nonumber
\end{aligned}
\end{equation}
\end{thm}
\begin{proof}
Let ${\mathbb{IF}}_{\tilde{\mathscr{A}}} = \langle \overline{A}, \overline{B} \rangle$ be an \emph{implication-based} \emph{intuitionistic fuzzy} \emph{kernel} of $\mathbb{IS}$.
Let $\alpha, \beta, \gamma, \gamma \in \Omega $ , $ \xi \in {\Delta}^{\ast} $.
The proof is by the method of induction on $ord(x) = n$.
If If $n = 0$ then $\xi = \textbf{0}$. 
\begin{equation}
\begin{aligned}
&\vDash_{\lambda} ((\beta \kappa,\xi,\alpha) \in A^{\ast}) \wedge ((\beta,\xi,\gamma) \in A^{\ast}) \wedge (\kappa \in \overline{A}) \\ \nonumber
&\rightarrow ((\beta \kappa,\textbf{0},\alpha) \in A^{\ast}) \wedge ((\beta,\textbf{0},\gamma) \in A^{\ast}) \wedge (\kappa \in \overline{A}) \\ \nonumber
&\rightarrow 0 \wedge 0 \wedge (\kappa \in \overline{A}) \\ \nonumber
&\rightarrow 0 \\ \nonumber
&\rightarrow (\alpha \gamma^{-1} \in \overline{A})
\end{aligned}
\end{equation}
\begin{equation}
\begin{aligned}
\vDash_{\lambda} (\alpha \gamma^{-1} \in \overline{B}) 
&\rightarrow 1 \\ \nonumber
&\rightarrow 1 \vee 1 \vee (\kappa \in \overline{B}) \\ \nonumber
&\rightarrow ((\beta \kappa,\textbf{0},\alpha) \in B^{\ast}) \vee ((\beta,\textbf{0},\gamma) \in B^{\ast}) \vee (\kappa \in \overline{B}) \\ \nonumber
&\rightarrow ((\beta \kappa,\xi, \alpha) \in B^{\ast}) \vee ((\beta ,\xi, \gamma) \in B^{\ast}) \vee (\kappa \in \overline{B}) \nonumber
\end{aligned}
\end{equation}
Thus the result holds for $n = 0$.
Assume that the result holds for all $y \in {\Delta}^{\ast} $ such that $ord(\psi) = n - 1$ and $n > 0$. 
Let $ \xi \in {\Delta}^{\ast} $ be such that $\xi = \psi \odot \omega$ where $\psi \in {\Delta}^{\ast} $ and  $\omega \in \Delta $ with $ord(\psi) = n - 1$ and $n > 0$. 
\setlength{\abovedisplayskip}{0pt}
\setlength{\belowdisplayskip}{0pt}
\begin{equation}
\begin{aligned}
&\vDash_{\lambda} ((\beta \kappa,\xi,\alpha) \in A^{\ast}) \wedge ((\beta,\xi,\gamma) \in A^{\ast}) \wedge (\kappa \in \overline{A}) \\ \nonumber
&\rightarrow ((\beta \kappa,\psi \odot \omega,\alpha) \in A^{\ast}) \wedge ((\beta,\psi \odot \omega, \gamma) \in A^{\ast}) \wedge  (\kappa \in \overline{A}) \\ \nonumber
&\rightarrow \left( \exists \zeta \{ ((\beta \kappa,\psi, \zeta) \in A^{\ast}) \wedge ((\zeta,\omega, \alpha) \in A) \}; \zeta \in \Omega \right) \wedge \left(  \exists \eta \{ ((\beta,\psi,\eta) \in A^{\ast}) \wedge \right.  \\ \nonumber
&\mbox{  } \hspace{2 cm} \left. ((\eta,\omega,\gamma) \in A) \};  \gamma \in \Omega \right) \wedge (\kappa \in \overline{A})  
\end{aligned}
\end{equation}
\begin{equation}
\begin{aligned}
&\rightarrow \left( \exists \eta \{ \exists \zeta \{ ((\beta \kappa,\psi,\alpha) \in A^{\ast}) \wedge ((\alpha,\omega,\alpha) \in A) \wedge (\kappa \in \overline{A})  \wedge ((\beta, \psi, \eta) \in A^{\ast}) \wedge \right. \\ \nonumber
&\mbox{  } \hspace{2 cm} \left.  ((\eta, \omega,\gamma) \in A) \}; \zeta \in \Omega \}; \eta \in \Omega \right)  \\ \nonumber
&\rightarrow \left( \exists \eta \{ \exists \zeta \{ ((\beta \kappa,\psi,\alpha) \in A^{\ast}) \wedge ((\beta,\psi,\eta) \in A^{\ast}) \wedge (\kappa \in \overline{A}) \wedge ((\zeta,\omega,\alpha) \in A)  \wedge \right. \\ \nonumber
&\mbox{  } \hspace{2 cm} \left.  ((\eta, \omega,\gamma) \in A) \}; \zeta \in \Omega \}; \eta \in \Omega \right)  \\ \nonumber
&\rightarrow \left( \exists \eta \{ \exists \zeta \{(\zeta \eta^{-1} \in \overline{A}) \wedge ((\zeta,\omega,\alpha) \in A)  \wedge ((\eta,\omega,\gamma) \in A) \}; \zeta \in \Omega \}; \eta \in \Omega \right)  \\ \nonumber
&\rightarrow \left( \exists \eta \{ \exists \zeta \{((\eta,\omega,\gamma) \in A) \wedge (((\eta.\eta^{-1})\alpha,\omega,\alpha) \in A) \wedge (\alpha \eta^{-1} \in \overline{A}) \}; \zeta \in \Omega \}; \eta \in \Omega \right)  \\ \nonumber
&\rightarrow \left( \exists \eta \{ \exists \zeta \{((\eta,\omega, \gamma) \in A) \wedge ((\eta.(\alpha \eta^{-1}),\omega,\alpha) \in A) \wedge (\zeta \eta^{-1} \in \overline{A}) \}; \zeta \in \Omega \}; \eta \in \Omega \right) \\ \nonumber
&\rightarrow \left( \exists \eta \{ \exists \zeta \{ ((\eta.(\zeta \eta^{-1}),\omega,\alpha) \in A) \wedge ((\eta,\omega,\gamma) \in A) \wedge (\zeta \eta^{-1} \in \overline{A}) \}; \zeta \in \Omega \}; \eta \in \Omega \right) \\ \nonumber
&\rightarrow (\alpha \gamma^{-1} \in \overline{A}) \nonumber
\end{aligned}
\end{equation}
\setlength{\abovedisplayskip}{0pt}
\setlength{\belowdisplayskip}{0pt}
\begin{equation}
\begin{aligned}
&\vDash_{\lambda} (\alpha \gamma^{-1} \in \overline{B})  \\ \nonumber
&\rightarrow \left( \forall \eta \{ \forall \zeta \{ ((\eta.(\zeta \eta^{-1}),\omega,\alpha) \in B)  \vee ((\eta,\omega,\gamma) \in B) \vee  (\zeta \eta^{-1} \in \overline{B}) \};  \zeta \in \Omega \}; \eta \in \Omega \right)  \\ \nonumber
&\rightarrow \left( \forall \eta \{ \forall \zeta \{((\eta,\omega,\gamma) \in B) \vee ((\eta.(\zeta \eta^{-1}),\omega,\alpha) \in B) \vee (\zeta \eta^{-1} \in \overline{B}) \}; \zeta \in \Omega \}; \eta \in \Omega \right) \\
&\rightarrow \left( \forall \eta \{ \forall \zeta \{((\eta,\omega,\gamma) \in B) \vee (((\eta.\eta^{-1})\zeta,\omega,\alpha) \in B) \vee (\zeta \eta^{-1} \in \overline{B}) \}; \zeta \in \Omega \}; \eta \in \Omega \right)  \\
&\rightarrow \left( \forall \eta \{ \forall \zeta \{(\zeta \eta^{-1} \in \overline{B}) \vee ((\zeta,\omega,\alpha) \in B)  \vee ((\eta,\omega,\gamma) \in B) \}; \zeta \in \Omega \}; \eta \in \Omega \right)  \\
&\rightarrow \left( \forall \eta \{ \forall \zeta \{ ((\beta \kappa,\psi,\alpha) \in B^{\ast}) \vee ((\beta,\psi,\eta) \in B^{\ast})  \vee (\kappa \in \overline{B}) \vee ((\zeta,\omega,\alpha) \in B)  \vee  \right. \\ \nonumber
&\mbox{  } \hspace{3 cm}  \left.  ((\eta,\omega,\gamma) \in B) \}; \zeta \in \Omega \}; \eta \in \Omega \right) \\ \nonumber
&\rightarrow \left( \forall \eta \{ \forall \zeta \{ ((\beta \kappa,\psi,\alpha) \in B^{\ast}) \vee ((\zeta,\omega,\alpha) \in B) \vee (\kappa \in \overline{B}) \vee ((\beta,\psi,\eta) \in B^{\ast}) \vee    \right. \\ \nonumber
&\mbox{  } \hspace{3 cm} \left.  ((\eta,\omega,\gamma) \in B) \}; \zeta \in \Omega \}; \eta \in \Omega \right)\\
&\rightarrow \left( \forall \zeta \{ ((\beta \kappa,\psi,\alpha) \in B^{\ast}) \vee ((\zeta,\omega,\alpha) \in B) \}; \zeta \in \Omega \right) \vee \left(  \forall \eta \{ ((\beta,\psi,\eta) \in B^{\ast}) \vee \right. \\ \nonumber
&\mbox{  } \hspace{3 cm} \left. ((\eta,\omega,\gamma) \in B) \}; \eta \in \Omega \right) \vee (\kappa \in \overline{B})  \\
&\rightarrow ((\beta \kappa,\psi \odot \omega,\alpha) \in B^{\ast}) \vee ((\beta,\psi \odot \omega,\gamma) \in B^{\ast}) \vee (\kappa \in \overline{B}) \\
&\rightarrow ((\beta \kappa,\xi,\alpha) \in B^{\ast}) \vee ((\beta,\xi,\gamma) \in B^{\ast}) \vee (\kappa \in \overline{B}) \nonumber
\end{aligned}
\end{equation}
The converse is trivial. 
\end{proof}
\begin{thm}
An implication-based intuitionistic fuzzy kernel ${\mathbb{IF}}_{\tilde{\mathscr{A}}} = \langle \overline{A}, \overline{B} \rangle$ of the implication-based intuitionistic fuzzy semiautomaton $\mathbb{IS} = \left(\Omega, \Delta, {\mathbb{IF}}_{\mathscr{A}} \right)$ over the finite group $(\Omega,\cdot)$ is an implication-based intuitionistic fuzzy subsemiautomaton if and only if 
\begin{enumerate}[(i)]
\item $\vDash_{\lambda} ((\epsilon,\xi,\alpha) \in A) \wedge (\epsilon \in \overline{A}) \rightarrow (\alpha \in \overline{A})$
\item $\vDash_{\lambda} (\alpha \in \overline{B}) \rightarrow ((\epsilon,\xi,\alpha) \in B) \vee (\epsilon \in \overline{B}) \hspace{3 cm }  \forall \alpha \in \Omega, \xi \in \Delta$
\end{enumerate}
where $'\epsilon'$ is the identity element of the finite group $(\Omega,\cdot)$.
\end{thm}
\begin{proof}
Let the \emph{implication-based intuitionistic fuzzy kernel} ${\mathbb{IF}}_{\tilde{\mathscr{A}}} = \langle \overline{A}, \overline{B} \rangle$ of $\mathbb{IS}$ be an \emph{implication-based} \emph{intuitionistic fuzzy} \emph{subsemiautomaton}. \\
Let $\alpha \in \Omega, \omega \in \Delta$.
\setlength{\abovedisplayskip}{0pt}
\setlength{\belowdisplayskip}{0pt}
\begin{equation}
\begin{aligned} 
\Rightarrow &\vDash_{\lambda} ((\alpha,\xi,\beta) \in A) \wedge (\alpha \in \overline{A}) \rightarrow (\beta \in \overline{A}) \\ \nonumber
&\vDash_{\lambda} (\beta \in \overline{B}) \rightarrow ((\alpha,\xi,\beta) \in B) \vee (\alpha \in \overline{B}) \nonumber
\end{aligned}  
\end{equation}
By definition
\begin{equation}
\begin{aligned} 
\Rightarrow &\vDash_{\lambda} ((\epsilon,\xi,\alpha) \in A) \wedge (\epsilon \in \overline{A}) \rightarrow (\alpha \in \overline{A}) \\ \nonumber
&\vDash_{\lambda} (\alpha \in \overline{B}) \rightarrow ((\epsilon,\xi,\alpha) \in B) \vee (\epsilon \in \overline{B}) \nonumber
\end{aligned}
\end{equation}
Conversely, let ${\mathbb{IF}}_{\tilde{\mathscr{A}}} = \langle \overline{A}, \overline{B} \rangle$ be an \emph{implication-based intuitionistic fuzzy kernel} of $\mathbb{IS}$ satisfying the conditions
\begin{enumerate}[(i)]
\item $\vDash_{\lambda} ((\epsilon,\xi,\alpha) \in A) \wedge (\epsilon \in \overline{A}) \rightarrow (\alpha \in \overline{A})$
\item $\vDash_{\lambda} (\alpha \in \overline{B}) \rightarrow ((\epsilon,\xi,\alpha) \in B) \vee (\epsilon \in \overline{B})  \hspace{3 cm }  \forall \alpha \in \Omega, \xi \in \Delta $
\end{enumerate}
Let $\alpha, \beta \in \Omega$ and $ \xi \in \Delta$.
\setlength{\abovedisplayskip}{0pt}
\setlength{\belowdisplayskip}{0pt}
\begin{equation}
\begin{aligned} 
&\vDash_{\lambda}((\epsilon,\xi,\gamma) \in A) \wedge ((\beta,\xi,\gamma) \in A) \rightarrow ((\epsilon \beta,\xi,\gamma) \in A) \\ 
&\Rightarrow \mbox{     } \vDash_{\lambda}((\epsilon \beta,\xi,\gamma) \in A) \rightarrow ((\epsilon,\xi,\gamma) \in A) \\
&\mbox{ Also } \mbox{     } \vDash_{\lambda}(\gamma \in \overline{A}) \rightarrow (\epsilon \in \overline{A})
\end{aligned}
\end{equation}
\setlength{\abovedisplayskip}{0pt}
\setlength{\belowdisplayskip}{0pt}
\begin{equation}
\begin{aligned} 
&\vDash_{\lambda} ((\epsilon \beta,\xi,\gamma) \in B) \rightarrow ((\epsilon,\xi,\gamma) \in B) \vee ((\beta,\xi,\gamma) \in B)  \\ 
&\Rightarrow \mbox{     } \vDash_{\lambda}((\epsilon,\xi,\gamma) \in B) \rightarrow ((\epsilon \beta,\xi,\gamma) \in B)   \\
&\mbox{ Also } \mbox{     } \vDash_{\lambda}(\epsilon \in \overline{B}) \rightarrow (\gamma \in \overline{B})
\end{aligned}
\end{equation}
Now
\setlength{\abovedisplayskip}{0pt}
\setlength{\belowdisplayskip}{0pt}
\begin{equation}
\begin{aligned}
&\vDash_{\lambda} ((\beta,\xi,\alpha) \in A) \wedge (\beta \in \overline{A}) \\ \nonumber
&\rightarrow ((\beta,\xi,\alpha) \in A) \wedge ((\epsilon,\xi,\gamma) \in A) \wedge (\epsilon \in \overline{A})  \wedge (\beta \in \overline{A}) \mbox{ by (1) }  \\ \nonumber
&\rightarrow ((\epsilon \beta,\xi,\alpha) \in A) \wedge ((\epsilon,\xi,\gamma) \in A) \wedge (\beta \in \overline{A}) \wedge (\epsilon \in \overline{A})  \\
&\rightarrow (\alpha \gamma^{-1} \in \overline{A}) \wedge (\gamma \in \overline{A}) \\
&\mbox{ Since } {\mathbb{IF}}_{\tilde{\mathscr{A}}} \mbox{ is an \emph{implication-based intuitionistic }} \mbox{\emph{fuzzy kernel}} \\
&\rightarrow (\alpha \gamma^{-1}.\gamma \in \overline{A}) \\
&\rightarrow (\alpha \in \overline{A}) 
\end{aligned}
\end{equation}
\setlength{\abovedisplayskip}{0pt}
\setlength{\belowdisplayskip}{0pt}
\begin{equation}
\begin{aligned}
&\vDash_{\lambda} (\alpha \in \overline{B}) \\ \nonumber
&\rightarrow (\alpha \gamma^{-1}.\gamma \in \overline{B}) \\ \nonumber
&\rightarrow (\alpha \gamma^{-1} \in \overline{B}) \vee (\gamma \in \overline{B}) \\
&\rightarrow ((\epsilon \beta,\xi,\alpha) \in B) \vee ((\epsilon,\xi,\gamma) \in B) \vee (\beta \in \overline{B}) \vee (\epsilon \in \overline{B}) \\
&\rightarrow ((\epsilon \beta,\xi,\alpha) \in B) \vee ((\epsilon,\xi,\gamma) \in B) \vee (\epsilon \in \overline{B}) \vee (\beta \in \overline{B}) \\
&\rightarrow ((\beta,\xi,\alpha) \in B) \vee ((\epsilon,\xi,\gamma) \in B) \vee (\beta \in \overline{B}) \vee (\epsilon \in \overline{B}) \\
&\rightarrow ((\beta,\xi,\alpha) \in B) \vee (\beta \in \overline{B})
\end{aligned}
\end{equation}
Therefore ${\mathbb{IF}}_{\tilde{\mathscr{A}}} = \langle \overline{A}, \overline{B} \rangle$ is an \emph{implication-based intuitionistic fuzzy subsemiautomaton} of $\mathbb{IS}$.
\end{proof}

\section{Conclusion}
In this paper, a subsystem namely \emph{implication-based intuitionistic fuzzy semiautomaton} (IB-IFSA) of a finite group has been constituted. Further using the perception of \emph{implication-based intutioinistic fuzzy subgroup} and \emph{implication-based intuitionistic fuzzy normal subgroup} of a finite group, \emph{implication-based intuitionistic fuzzy kernel} and \emph{implication-based intuitionistic fuzzy subsemiautomaton} of an IB-IFSA are developed. The necessary and sufficient conditions for \emph{implication-based intuitionistic fuzzy kernel} and \emph{implication-based intuitionistic fuzzy subsemiautomaton} of an IB-IFSA are examined. 

\bibliography{mybibfile}

\end{document}